\documentclass[leqno]{amsart}

\usepackage{latexsym,amsmath,amsthm}
\usepackage{amssymb,bm}
\usepackage{xcolor}

\usepackage{lineno,hyperref}

\usepackage{comment}
\newcommand{\R}{{\mathbb R}}
\newcommand{\m}{\mathfrak{m}}
\newcommand{\cc}{\mathfrak{c}}

\def\ba{{\bm a}}
\def\diam{{\operatorname{diam\,}}}
\def\div{{\operatorname{div\,}  }}
\def\llangle{{\langle\!\langle}}
\def\rrangle{{\rangle\!\rangle}}

\newtheorem{thm}{Theorem}[section]  
\newtheorem{lem}[thm]{Lemma}	       
  
\newtheorem{remark}{Remark} 

\numberwithin{equation}{section}

\begin{document}

\title[Conormal problem for elliptic equations]{Conormal problem  for a kind of  nonlinear  elliptic equations}

\author[]{A.F. Buscetto}
\address{Department of Mathematics,
University of Salerno, Italy}
\email{a.buscetto3@studenti.unisa.it}

\author[]{R. Rescigno}
\address{Department of Mathematics,
University of Salerno, Italy}
\email{rorescigno@unisa.it}

\author[]{L.G.  Softova}
\address{Department of Mathematics,
University of Salerno, Italy}
\email{lsoftova@unisa.it}

\author[]{S. Tramontano}
\address{Department of Mathematics,
University of Salerno, Italy}
\email{stramontano@unisa.it}



\date{}

\begin{abstract}
We study the global boundedness of weak solutions to a class   of nonlinear elliptic equations in bounded Lipschitz domains, subject to   conormal  boundary conditions. Using techniques from  Morrey space theory, Adams-Maz'ya  trace inequalities, and higher integrability properties of the  gradient, we derive  uniform a priori bounds  for the solutions. A crucial  step in our analysis  is the construction of appropriate  measures associated with the solution  and the application of the Hartman–Stampacchia lemma, which enables us  to   control  the solution up to  the boundary of the domain.  
\end{abstract}

\maketitle

\section{Introduction}

The main objective  of this paper is to establish the {\it essential boundedness} of weak solutions 
$u \in W^{1,2}(\Omega)$ to the {\it conormal derivative problem} for a class of second-order elliptic equations in divergence form
\begin{equation}\label{1.1}
\begin{cases}
\div \ba(x,u,Du) = b(x,u,Du), 
& x \in \Omega, \\[6pt]
 \ba(x,u,Du)\cdot \bm{\nu}(x) = \psi(x,u), 
& x \in \partial\Omega.
\end{cases}
\end{equation}
Here,  $\Omega \subset {\R}^n$, $n \geq 2$, is a bounded domain with Lipschitz  boundary, and 
$\boldsymbol{\nu}(x) = (\nu_1(x),\ldots,\nu_n(x))$ denotes the outward unit normal vector at $x \in \partial\Omega$.

The nonlinear terms
$$
\ba(x,z,\xi)=(a_1(x,z,\xi),\ldots,a_n(x,z,\xi)),
\quad b(x,z,\xi),
\quad \psi(x,z)
$$
are assumed to be \textit{Carath\'eodory functions}. Namely, $\ba$ and $b$ are measurable in
$x\in\Omega$ for every $(z,\xi)\in\mathbb R\times\mathbb R^n$ and continuous in
$(z,\xi)$ for almost every $x\in\Omega$, while $\psi$ is measurable in
$x\in\partial\Omega$ for every $z\in\mathbb R$ and continuous in $z$ for almost every
$x\in\partial\Omega$.

The boundedness and regularity of weak solutions to quasilinear elliptic equations with conormal boundary conditions have  been  studied in the literature under various assumptions. 
In the sub-controlled growth setting, global boundedness was established  by Lieberman 
\cite{Lieb1,Lieb2} (see also \cite{Lieb3}), with related contributions by  
Winkert  \cite{Win}.

Under the stronger \emph{controlled growth} assumptions, classical results of  Ladyzhenskaya and Ural'tseva
\cite[Chapter~X]{LU} ensure the  essential boundedness and H\"older continuity of weak solutions, provided that the $x$-dependence of the coefficients satisfies suitable Lebesgue integrability conditions; see also \cite{Kim}.
Further results were  obtained  by Arkhipova \cite{Arkh1,Arkh2}, who established reverse H\"older inequalities for the gradients of solutions to quasilinear equations with  conormal boundary conditions.

The present work starts  from the classical results in the  Lebesgue spaces,  assuming that the $x$-dependence of the nonlinearities in \eqref{1.1} belongs to appropriate \emph{Morrey spaces}.

This framework allows us to treat coefficients exhibiting  stronger local singularities. 
By combining Morrey space estimates, Adams-Maz'ya trace inequalities, and higher integrability properties of gradients, we derive uniform \emph{a priori} bounds for weak solutions.
A crucial role in the analysis is played by the construction of suitable Radon  measures associated with the solution and the application of the Hartman--Stampacchia lemma, which yields effective control of the solution  in the whole domain up to the boundary.

In this paper, we extend several results on the boundedness and
Morrey regularity of solutions to Dirichlet boundary value problems for nonlinear
elliptic and parabolic operators satisfying structural coercivity  conditions;
see \cite{AFPS,AFS,PS1,PS2,Sf1,Sf2} and the references therein.

\section{Structural Assumptions}

The divergence-form operator in \eqref{1.1} is assumed to satisfy the following
\textit{coercivity condition}: there exist constants $\gamma>0$ and $\Lambda>0$
such that
\begin{equation}\label{1.3} 
    \ba(x,z,\xi)\cdot \xi\geq \gamma |\xi|^2
    -\Lambda|z|^{2^*}-\Lambda\varphi_1(x)^2
\end{equation}
for almost every  $x \in \Omega$ and all $(z,\xi) \in \R\times \R^n$,  where  $2^*$ is the Sobolev conjugate exponent of 2. 
Recall that, for every $p\ge 1$, the Sobolev conjugate exponent $p^*$ is defined by
$$
p^*= \begin{cases}
\frac{np}{n-p}, & \text{ if } p<n,
\\[8pt]
\text{any exponent in } (p,\infty), & \text{ if } p\geq  n.
\end{cases}
$$

In addition to the  coercivity assumption, we impose the following {\it controlled growth conditions} on the nonlinear operators and the boundary term: 
\allowdisplaybreaks
\begin{align}\label{1.4}
|\ba(x,z,\xi)|&\leq  \Lambda \Big(\varphi_1(x)  +|z|^{\frac{2^*}{2}}+|\xi|\Big),\\[4pt] 
\label{1.6}
|b(x,z,\xi)|&\leq \Lambda \Big(\varphi_2(x)+|z|^{2^*-1} +|\xi|^{\frac{2(2^*-1)}{2^*}}\Big),\\[4pt]
\label{1.8}
|\psi(x,z)|&\leq \psi_1(x)+\psi_2(x)|z|^\beta .
\end{align}
Here
\begin{equation}\label{1.5}
\begin{cases}
\varphi_1\in L^{p_1,\lambda_1}(\Omega), & p_1>\max\left\{2, n-\lambda_1\right\}, \  \lambda_1\in (0,n), \\[4pt]
\varphi_2\in L^{p_2,\lambda_2}(\Omega), & p_2>\max\left\{\dfrac{2n}{n+2}, \dfrac{n-\lambda_2}{2}\right\}, \   \lambda_2\in (0,n)\\[4pt]
   \psi_1\in L^{q_1}(\partial \Omega) & q_1>n-1,\\[4pt]  
\psi_2\in L^{q_2}(\partial  \Omega),& q_2> \dfrac{2(n-1)}{n-\beta(n-2)},\quad \beta\in\Big[0,\frac{n}{n-2}\Big).
\end{cases}
\end{equation}
 In the case  $n=2,$ we assume   $\beta\in[0,+\infty).$

For the reader's convenience, we  recall that the Morrey space $L^{p,\lambda}(\Omega)$, with 
$p \in [1,\infty)$ and $\lambda \in (0,n)$, is defined as the set  of all  functions $u \in L^p(\Omega)$ such that
$$
\|u\|_{L^{p,\lambda}(\Omega)}
:= \sup_{x\in\Omega,\ 0<\rho<\diam\Omega}
\left( \rho^{-\lambda} \int_{\Omega_\rho(x)} |u(y)|^p\, dy \right)^{1/p} < \infty,
$$
where $\Omega_\rho(x) := \Omega \cap B_\rho(x)$  
and $B_\rho(x)$ denotes an  open ball in $\R^n$ centered at $x \in \Omega$ and of radius $\rho,$ 
while $\diam\Omega$ stands for the diameter of $\Omega$.

A function $u \in W^{1,2}(\Omega)$ is called a {\it weak solution} to \eqref{1.1} if it   satisfies the  identity
\begin{equation}\label{1.2}
\begin{split}
\int_\Omega \ba\big(x,u(x),Du(x)\big)\cdot Dv(x)\,dx&+\int_\Omega b\big(x,u(x),Du(x)\big)v(x) \, dx\\
     &=\int_{\partial\Omega} \psi\big(x,u\big)v(x)\, d\sigma_x
\end{split}
\end{equation}
for every test function   $v \in W^{1,2}(\Omega).$
The {\it controlled  growth conditions} \eqref{1.4}-\eqref{1.8} ensure  that all terms in
\eqref{1.2} are well defined, and therefore  the above definition is   well posed.

Throughout this  paper, the term 
\textit{known quantities} refers to constants depending only on  $n,\Lambda,\gamma,$  the exponents and norms of the  data  appearing  in  \eqref{1.5},  the diameter of 
 $\Omega$ and the Lipschitz character of   $\partial\Omega.$

Under these assumptions, our main result
(Theorem~\ref{thm3.1}) asserts that every weak solution  $u \in W^{1,2}(\Omega)$ to  problem \eqref{1.1} is essentially bounded, with the bound depending only on known quantities and the norm $\|Du\|_{L^2(\Omega)}.$  A  similar result for quasilinear Dirichlet problems was  established  in \cite{BPS}. We also refer to the work of 
 Nazarov and Ural'tseva \cite{NU}, where local properties such  as the strong maximum principle, the  Harnack inequality, and H\"older continuity were 
 established for weak solutions of linear divergence-form  equations with 
 lower-order coefficients belonging to Morrey-type spaces.

In the present paper,  we  extend the results obtained in \cite{AFPS}, where the authors considered a  homogeneous conormal problem under  a sign condition on the boundary data. 
Using  the  De Giorgi method, the technique of  Ladyzhenskaya and Ural'tseva  (cf. \cite{LU}),  and the Adams-Maz'ya trace inequalities  \cite{Ad2,Maz1,Maz},  we derive  decay estimates for 
the measure of the super-level sets of the solution.

In addition,
the {\it higher  gradient integrability} results of Arkhipova \cite{Arkh1,Arkh2} enable   us to estimate   the energy of $u$ with respect to a suitable Radon measure. 
Combined with the Hartman--Stampacchia lemma \cite{AFPS,HS,LU,Sf1}, these estimates yield the essential boundedness of weak solutions.

\section{Auxiliary Results}

In this section we collect several auxiliary results that will be used in the proof of the  main theorem.

\begin{lem}[Embeddings between Morrey spaces, \cite{Pic}]\label{lem2.1}
Let $s',s''\in [1,\infty)$ and $\theta',\theta''\in [0,n)$. Then
$$
L^{s',\theta'}(\Omega) \subseteq L^{s'',\theta''}(\Omega)
\quad \text{if and only if} \quad
s' \ge s'' \ge 1, \quad \frac{s'}{n-\theta'} \ge \frac{s''}{n-\theta''}.
$$
\end{lem}

 We next state a \textit{multiplicative}  version of the \textit{Gagliardo-Nirenberg interpolation inequality} together with a trace inequality that will be used to estimate the boundary term in \eqref{1.2} (cf. \cite[Theorem~1.4.8/1]{Maz}, \cite{Maz1}). To this end, we introduce the notation 
$$
\llangle v \rrangle_{\sigma,\Omega} := \Big( \int_\Omega |v(x)|^\sigma \, dx \Big)^{1/\sigma},
\qquad  \ \sigma>0.
$$

\begin{lem}[Gagliardo--Nirenberg inequality]\label{G-N}
Let $\Omega$ be a bounded domain having  the cone property and let $v\in W^{1,p}(\Omega), p>1$.  
Then there exists a constant  $C>0$ depending only on
$n,p,\sigma,\delta$, and $\Omega$, such that
\begin{equation}\label{eq-GN2}
\llangle v \rrangle_{r,\Omega} \leq C \big( \llangle Dv \rrangle_{p,\Omega} + \llangle v \rrangle_{\sigma,\Omega} \big)^\delta \llangle v \rrangle_{\sigma,\Omega}^{1-\delta},
\end{equation}
for every   $\delta\in[0,1]$, where
$$
\frac1r=
\delta\frac{n-p}{np}+\frac{1-\delta}{\sigma}. 
$$
If  $n=p$,  \eqref{eq-GN2} holds for every $\delta\in[0,1)$.
\end{lem}

\begin{remark}\label{rem1}

The seminorm $
\llangle v\rrangle_{\sigma,\Omega},
$ $ \sigma>0, $
satisfies the estimate
$$
\llangle u+v\rrangle_{\sigma,\Omega}
\leq  C_\sigma \bigl(
\llangle u\rrangle_{\sigma,\Omega} +
\llangle v\rrangle_{\sigma,\Omega}
\bigr),
$$
where
$$
C_\sigma=
\begin{cases}
1, & \sigma\ge1,\\[4pt]
2^{\frac1\sigma-1}, & 0<\sigma<1.
\end{cases}
$$

Indeed, if $\sigma\ge1$, the estimate follows directly from the
Minkowski inequality.

If $0<\sigma<1$, the function $t\mapsto t^\sigma$ is concave, and therefore
$$
|u+v|^\sigma \leq |u|^\sigma+|v|^\sigma.
$$
Integrating over $\Omega$, we obtain
$$
\llangle u+v\rrangle_{\sigma,\Omega}^\sigma
\leq
\llangle u\rrangle_{\sigma,\Omega}^\sigma +
\llangle v\rrangle_{\sigma,\Omega}^\sigma.
$$
Finally, since the function $t\mapsto t^{1/\sigma}$ is convex, 
$$ 
(a+b)^{1/\sigma} \leq 2^{\frac1\sigma-1}
\bigl(a^{1/\sigma}+b^{1/\sigma}\bigr), 
$$ 
which yields the desired estimate.

\end{remark}

The following  \textit{Maz'ya  trace} inequality  (cf. \cite[Corollary 1.4.7/2]{Maz}) plays    a crucial role in  estimating    the boundary  integral in \eqref{1.2}.
\begin{lem}\label{trace}
Let $\Omega\subset \R^n, n\geq 2$ be a bounded Lipschitz domain and let $\mathfrak{m}$ be a measure in $\overline{\Omega}$ satisfying
\begin{equation}\label{eq-mball}
\sup_{x\in \overline{\Omega},\ \rho\in(0,1)} \rho^{-s}\mathfrak{m}(B_\rho(x)) \leq K 
\end{equation}
for some $s\in(n-p,n]$ and $K>0$. Then there exists a positive constant $C=C(n,s,p,r,K,\Omega)$ such that
\begin{equation}\label{eq-trace}
\|v\|_{L^r(\mathfrak{m},\overline{\Omega})} \leq C \|v\|_{W^{1,p}(\Omega)}^\tau \|v\|_{L^p(\Omega)}^{1-\tau},
\end{equation}
for all  $v\in W^{1,p}(\Omega)$, $p>1$, where  
$$
r \in \Big[p,\,  \frac{sp}{n-p}\Big), \qquad    \tau = \frac{n}{p} - \frac{s}{r} <1.
$$ 
\end{lem}

In the endpoint case  $\tau=1,$ that is, when   $s=\frac{r}{p}(n-p)$,  Adams proved in  \cite{Ad2}  (see also \cite{Ad1,Ch,Sf1}) that, under  the hypothesis \eqref{eq-mball},  we have
\begin{equation}\label{eq-Ad}
\|v\|_{L^r(\mathfrak{m},\overline{\Omega})} \leq C \|Dv\|_{L^p(\Omega)}, \qquad \forall \, v\in W_0^{1,p}(\Omega), \ p\in(1,n).
\end{equation}

For Lipschitz domains, which have  the extension property, every  function $v\in W^{1,p}(\Omega)$ 
admits an extension $V\in W^{1,p}(\R^n).$ Choosing a cut-off function $\zeta\in C_0^\infty(\R^n),$ such that $\zeta=1$ on $\Omega,$ we may apply   \eqref{eq-Ad} to the function $V\zeta$ (cf. \cite{AFPS}). 
Consequently, the  Adams trace inequality takes the form 
$$
\|v\|_{L^r(\mathfrak{m},\overline{\Omega})} \leq C \|v\|_{W^{1,p}(\Omega)}, \qquad \forall \, v\in W^{1,p}(\Omega), p\in(1,n).
$$

Interpolating the $L^p(\Omega)$ norm appearing  on the 
right-hand side by  means of the Gagliardo--Nirenberg inequality \eqref{eq-GN2} with $r=p$, we obtain the following estimate.

\begin{lem}[Adams trace inequality]\label{lem2.2}
Let $\mathfrak{m}$ be a positive Radon measure on $\overline\Omega$   satisfying \eqref{eq-mball} with $s=\frac{r}{p}(n-p),$ $1<p<r<\infty$.
Then for every  $v\in W^{1,p}(\Omega)$ and every  $\sigma \in (0,p]$, there exists a constant  $C>0$ such that
$$
\|v\|_{L^r(\m;\overline{\Omega})} \leq C \big( \|Dv\|_{L^p(\Omega)} + \llangle v \rrangle_{\sigma,\Omega} \big).
$$

In particular, if $d\m = \cc(x)\,  dx$ with $\cc\in L^{1,n-p+\epsilon_0}(\Omega)$ with  $\epsilon_0>0$ sufficiently small and $n>p$, then
$$
\Big( \int_\Omega |v(x)|^{\frac{p(n-p+\epsilon_0)}{n-p}} \,\cc(x)\, dx \Big)^{\frac{n-p}{p(n-p+\epsilon_0)}} \leq C \big( \|Dv\|_{L^p(\Omega)} + \llangle v \rrangle_{\sigma,\Omega} \big),
$$
and the constant depends only on $n,\sigma,p,K,$ the norm of $\cc(x)$,  and the Lipschitz constant of $\partial\Omega.$
\end{lem}

In the special  case where $\m$ coincides with  the Lebesgue measure on $\Omega,$ we have $s=n$ in \eqref{eq-mball}.  Hence $ \frac{sp}{n-p}=\frac{np}{n-p}=p^*$ and Lemma~\ref{lem2.2} yields  the following  Sobolev embedding (cf. \cite[Chapter~II, \S~2]{LU}).
\begin{lem}\label{lem2.3}
Let $\Omega\subset \R^n$ be a bounded Lipschitz domain. Then there exists a constant  $C>0$ such that
$$
\|v\|_{L^{p^*}(\Omega)} \leq C \big( \|Dv\|_{L^p(\Omega)} + \llangle v\rrangle_{\sigma,\Omega} \big), \quad \forall \,  v \in W^{1,p}(\Omega), \,p<n, \ \sigma\in(0,p].
$$
\end{lem}

The reverse H\"older inequality, combined   with the Gehring-Giaquinta lemma is a powerful tool for improving the integrability of weak solutions  and plays a  central role in  the regularity  theory of PDEs.

In \cite{Arkh2}, Arkhipova studies a conormal boundary value problem for  second-order quasilinear    elliptic systems in divergence form. 
She considers generalized solutions $\bm{u}\in W^{1,m}(\Omega;\R^N)$ where  $m\geq 2$ and  $N\geq 1.$ 
The coercivity and  structural conditions involve continuous functions of $\bm{u}$, which  in our setting  are replaced by  positive constants, while  the given data  belong to suitable Lebesgue spaces compatible with our regularity assumptions \eqref{1.5}.  
Moreover, 
the a priori  boundedness of the solution, allows one to assume   quadratic (and therefore optimal) growth   of the lower order term $\bm{b}$ with respect to the gradient.

After first  establishing  a Gehring-Giaquinta-type inequality (see \cite[Theorem 1]{Arkh1}), Arkhipova  proves the  higher integrability of the gradient of weak  solutions, provided that   the functions of $\bm{u}$ appearing in the structural conditions satisfy an additional assumption.
As observed in  \cite[Theorem 4]{Arkh2} and the corresponding remarks, this extra assumption is no longer needed when the growth of $\bm{b}$ with respect to the gradient is subquadratic.

It should be noted that, in our setting,  we do not assume the \textit{a priori} boundedness of the solution.  Consequently,  the growth of the lower order term $b$ with respect to the gradient is automatically restricted  to be  subquadratic,  and therefore Arkhipov's  
 \textit{ higher gradient  integrability }  result  applies  without any additional assumptions. 
 Similar results were also  obtained in   \cite{Kim} for  the case $m=2$.
 
\begin{lem}\label{lem2.4}
Assume \eqref{1.3}--\eqref{1.8}, and let $u\in W^{1,2}(\Omega)$ be a weak solution of \eqref{1.1}.  
Then there exists an exponent  $m_0>2$ such that 
$$
\|Du\|_{L^{m_0}(\Omega)} \leq C,
$$
where $C$  depends only on known quantities and on
$\|Du\|_{L^2(\Omega)}$.
\end{lem}

The following lemma, proved in \cite{HS} (see also \cite[Chapter~II, Lemma~5.1]{LU}), allows us to estimate the measure of the level sets of the solution.

\begin{lem}[Hartman--Stampacchia, \cite{HS,LU}]\label{lem2.5}
Let $\tau:\mathbb{R}\to[0,\infty)$ be a non-increasing function. 
Assume that there exist constants
$C>0$, $k_0\ge0$, $\delta>0$, and $\alpha\in[0,1+\delta]$ such that
$$
\int_k^\infty \tau(t)\,dt
\leq Ck^\alpha\tau(k)^{1+\delta},
\qquad  \forall\,k\geq k_0.
$$
Then there exists a constant $k_{\max}$, depending only on $C$, $k_0$, $\delta$, and $\alpha$, such that
$$
\tau(k)=0, \qquad \forall\ k\geq  k_{\max}.
$$
\end{lem}

\section{Global  Boundedness of Solutions}

We are now in a position to state the main result of this paper.

\begin{thm}\label{thm3.1}
Let $\Omega\subset\mathbb{R}^n$ be a bounded Lipschitz domain, and assume that conditions
\eqref{1.3}--\eqref{1.5} are satisfied.
Then every weak solution
$u\in W^{1,2}(\Omega)$ of \eqref{1.1} is globally bounded. More precisely, there exists a constant
$M>0$, depending only on the structural constants of the problem and on
$\|u\|_{W^{1,2}(\Omega)}$, such that
\begin{equation}\label{eq-max}
\|u\|_{L^\infty(\Omega)}\leq  M.
\end{equation}
\end{thm}

\begin{proof}

\textbf{The case $n>2$.} \quad
Without loss of generality, we extend the functions $\varphi_1$ and $\varphi_2$ by zero outside $\Omega$.
Then we  define a positive Radon measure $\mathfrak{m}$ on $\overline{\Omega}$ by
\begin{equation}\label{eq-measure}
d\mathfrak{m} :=
\left( \chi_\Omega(x) +\varphi_1(x)^2 +\varphi_2(x)
+|u(x)|^{\frac{4}{n-2}} \right)\,dx,
\end{equation}
where $\chi_\Omega$ denotes the characteristic function of $\Omega$.
Then, for every measurable set $E\subset\Omega$,
$$
\mathfrak{m}(E)
= \int_E
\left( 1+\varphi_1(x)^2+\varphi_2(x)+|u(x)|^{\frac{4}{n-2}} \right)\,dx.
$$
Hence, $\mathfrak{m}$ is a positive Radon measure that is absolutely continuous with respect to the Lebesgue measure.

Let $B_\rho\subset\mathbb{R}^n$ be a ball. Using assumption~\eqref{1.5},
Lemmas~\ref{lem2.4} and~\ref{lem2.3} (with $p=2$), and arguing as in
\cite{AFPS}, we obtain the following estimates:
\begin{equation}\label{eq-chi}
\int_{B_\rho} \chi_\Omega(x)\, dx\leq C(n)\, \rho^{n-2+2}; 
\end{equation}
\begin{equation}\label{eq-phi1}
\begin{split}
\int_{B_\rho}\varphi_1(x)^2\,dx
&\leq |B_\rho|^{1-\frac{2}{p_1}}
\left(\int_{B_\rho}\varphi_1(x)^{p_1}\,dx\right)^{\frac{2}{p_1}}\\
&\leq C(n)\, \rho^{n\left(1-\frac{2}{p_1}\right)}
\Big( \rho^{\lambda_1}\|\varphi_1\|_{L^{p_1,\lambda_1}(\Omega)}^{p_1}
\Big)^{\frac{2}{p_1}}\\
&= C(n)\, \|\varphi_1\|_{L^{p_1,\lambda_1}(\Omega)}^2\,
\rho^{n-2+\left(2-\frac{2(n-\lambda_1)}{p_1}\right)};
\end{split}
\end{equation}
\begin{equation}
\begin{split}\label{eq-phi2}
\int_{B_\rho}\varphi_2(x)\,dx
&\leq  |B_\rho|^{1-\frac1{p_2}}
\left(\int_{B_\rho}\varphi_2(x)^{p_2}\,dx
\right)^{\frac1{p_2}}\\
&\leq C(n)\,  \rho^{n\left(1-\frac1{p_2}\right)}
\Big( \rho^{\lambda_2}
\|\varphi_2\|_{L^{p_2,\lambda_2}(\Omega)}^{p_2}
\Big)^{\frac1{p_2}}\\
&= C(n)\,  \|\varphi_2\|_{L^{p_2,\lambda_2}(\Omega)}\,
\rho^{n-2+\left(2-\frac{n-\lambda_2}{p_2}\right)};
\end{split}
\end{equation}
\begin{equation}
\begin{split}\label{3.2}
\int_{B_\rho}|u(x)|^{\frac{4}{n-2}}\,dx
&\leq |B_\rho|^{1-\frac{4}{m_0^*(n-2)}}
\left( \int_{B_\rho}|u(x)|^{m_0^*}\,dx
\right)^{\frac{4}{m_0^*(n-2)}} \\
&\leq C(n)\, \rho^{n\left(1-\frac{4}{m_0^*(n-2)}\right)}
\|u\|_{L^{m_0^*}(\Omega)}^{\frac{4}{n-2}}\\
&=
C(n)\,
\|u\|_{L^{m_0^*}(\Omega)}^{\frac{4}{n-2}}\, 
\rho^{n-2+\left(2-\frac{4n}{m_0^*(n-2)}\right)}.
\end{split}
\end{equation}

The exponent in the last estimate is positive, since
$ 2-\frac{4n}{m_0^*(n-2)}>0, $
which follows from Lemma~\ref{lem2.4}, with  $m_0\in (2,n),$ and 
$$
m_0^*=\frac{nm_0}{n-m_0}>\frac{2n}{n-2}.
$$

Setting 
$$
\varepsilon_0:=\min\left\{2,\,
2-\frac{2(n-\lambda_1)}{p_1},\,
2-\frac{n-\lambda_2}{p_2},\,
2-\frac{4n}{m_0^*(n-2)}\right\},
$$
we conclude from \eqref{1.5} that $\varepsilon_0>0$. Hence,  condition \eqref{eq-mball} is satisfied  with  $s=n-2+\varepsilon_0$.  
Moreover, the constant $K$ in \eqref{eq-mball} depends only on $n,$ $\|\varphi_1\|_{L^{p_1,\lambda_1}(\Omega)}$, $\|\varphi_2\|_{L^{p_2,\lambda_2}(\Omega)},$ and $\|u\|_{L^{m_0^*}(\Omega)}.$  Consequently, Lemmas~\ref{trace} and~\ref{lem2.2} apply  to the measure $\mathfrak{m}$ with $p=2$.

Let  $k>1$ and define  the truncation
$$
u_k(x):=\max\{u(x)-k,0\},
\qquad
A_k:=\{x\in\Omega:\ u(x)>k\}.
$$
Since $u_k\in W^{1,2}(\Omega)$ and $u_k\equiv0$ on $\Omega\setminus A_k$, Hölder's inequality together with Lemma~\ref{lem2.2} (with $p=2$) yields
\begin{equation}\label{3.4}
\begin{split}
\int_{A_k}u_k(x)\,d\m  &\leq C(n)\, \m(A_k)^{1-\frac{n-2}{2(n-2+\varepsilon_0)}}
\left(
\int_{A_k}
u_k(x)^{\frac{2(n-2+\varepsilon_0)}{n-2}}
\,d\m
\right)^{\frac{n-2}{2(n-2+\varepsilon_0)}}
\\
&\leq C(n)\,\m (A_k)^{1-\frac{n-2}{2(n-2+\varepsilon_0)}}
\left(\|Du_k\|_{L^2(A_k)}+\llangle u_k\rrangle_{\sigma,A_k}\right)
\end{split}
\end{equation}
for every  $\sigma\in(0,2].$

 Taking $u_k$ as a test function in \eqref{1.2} and  using  the structural assumptions  \eqref{1.3},   \eqref{1.6} together with the estimate $0<\frac{u_k(x)}{|u(x)|}<1$,  and   Young's inequality (with  conjugate exponents $\frac{2n}{n+2}$ and $\frac{2n}{n-2}$), we obtain
 \begin{align*}
a(x,u,Du)\cdot Du_k&\geq \gamma|Du_k|^2-\Lambda |u|^{\frac{2n}{n-2}} -\Lambda \varphi_1(x)^2,\\
b(x,u,Du)u_k&\leq  |b(x,u,Du)u|\frac{u_k}{|u|}<\Lambda\big( \varphi_2(x)|u| + |u|^\frac{2n}{n-2} + |Du_k|^\frac{n+2}{n} |u|\big),\\
&\leq \Lambda\big( \varphi_2(x)|u| + \varepsilon\, C(n)|Du_k|^2  + C(\varepsilon, n)\,|u|^\frac{2n}{n-2} \big).
 \end{align*}

Substituting the above estimates into \eqref{1.2} and choosing $\varepsilon>0$ sufficiently small, we get
\begin{equation}\label{3.5}
\begin{split}
\int_{A_k} |Du_k(x)|^2&\,dx \leq C \Bigg( \underbrace{\int_{A_k} \varphi_1(x)^2\, dx}_{I_1} + \underbrace{\int_{A_k} \varphi_2(x)\, |u(x)|\, dx}_{I_2} \\
& + \underbrace{\int_{A_k} |u(x)|^2|u(x)|^{\frac{4}{n-2}}\,dx}_{I_3} + \underbrace{\int_{\partial \Omega \cap \overline{A}_k} \psi(x,u(x))\, u_k(x)\, d\sigma_x}_{J} \Bigg).
\end{split}
\end{equation}

The volume integrals   $I_1,$ $I_2,$ and  $I_3$  are estimated  as follows,
\begin{equation}\label{eq-I1}
I_1 \leq  \m(A_k) \leq k^2\, \m (A_k),\quad k>1,
\end{equation}
where the second inequality follows from the definition of the measure $\m.$ Furthermore,  H\"older's inequality,  Lemma~\ref{lem2.2},  and Young's inequality yield 
\begin{equation}\label{eq-I2}
\begin{split}
I_2& \leq  \int_{A_k} \varphi_2(x)(u(x)-k)\, dx +k\m(A_k)\leq 
\int_{A_k} u_k(x)\, d\m +k^2\m(A_k)\\
&\leq  \|u_k\|_{L^{\frac{2s}{n-2}}(\m,A_k)} \m(A_k)^{1-\frac{n-2}{2s}}  +k^2\m(A_k)\\
&\leq C \m(A_k)^{1-\frac{n-2}{2s}} \big(\| Du_k\|_{L^2(A_k)} +\llangle u_k\rrangle_{\sigma,A_k}  \big) + k^2\m(A_k)\\
 &\leq \varepsilon  \|Du_k\|_{L^2(A_k)}^2 +C \llangle u_k \rrangle^2_{\sigma, A_k}+ C k^2 \m (A_k) +C(\varepsilon) \m(A_k)^{2-\frac{n-2}{s}}.
 \end{split}
 \end{equation}
The last term is  estimated using the definition of the measure $\m$, the inclusion
$A_k\subset\Omega$,    the fact that $k>1$, and that $s\in(n-2,n]$. Indeed,
$$
\m(A_k)^{2-\frac{n-2}{s}}
\leq \m(\Omega)^{1-\frac{n-2}{s}}\m(A_k)
\leq Ck^2\,\m(A_k),
$$
where the constant $C$ depends only on the known quantities.

To estimate $I_3$ we apply  H\"older's inequality with respect to the measure $d\cc= \cc(x)\,dx$ where 
$$
\cc(x)=|u(x)|^{\frac{4}{n-2}}\in L^{1, n-2+ \epsilon_0}(\Omega), \quad  \epsilon_0=2-\frac{4n}{m_0^*(n-2)},
$$
as  follows from  \eqref{3.2}.
Consequently,   Lemma~\ref{lem2.2} yields
 \begin{equation}\label{eq-I3}
   \begin{split}  
 I_3 &  \leq \int_{A_k} u_k(x)^2\,  \cc(x)\, dx
 +k^2\m(A_k)\\
 &\leq \left( \int_{A_k} u_k(x)^{\frac{2(n-2+\epsilon_0)}{n-2}} \, \cc(x)\, dx\right)^{\frac{n-2}{n-2+\epsilon_0}} \left(\int_{A_k} \cc(x)\, dx \right)^{ \frac{\epsilon_0}{n-2+\epsilon_0}} \\
 &\leq C 
\cc(A_k)^{1- \frac{n-2}{n-2+\epsilon_0}} \left( \|Du_k\|_{L^2(A_k)}^2
+ \llangle u_k \rrangle^2_{\sigma,A_k}
\right) + k^2 \m (A_k),
\end{split}
\end{equation}
where $\cc(A_k):=\int_{A_k}\cc(x)\, dx.$

Combining  the  estimates above, we obtain
\begin{equation}\label{3.13}
I_1 + I_2 + I_3
\leq C \Big( (
\cc(A_k)^{ \frac{\epsilon_0}{n-2+\epsilon_0}}+ \varepsilon ) 
\|Du_k\|_{L^2(A_k)}^2
+ k^2 \m (A_k)
+ \llangle u_k \rrangle^2_{\sigma,A_k}\Big),
\end{equation}
for arbitrary $ \varepsilon > 0 $, every  $ k \geq 1 $, and every  $ \sigma \in (0,2] $.

Since the sets $A_k$ are nested, the  function $k\mapsto |A_k|$ is non-increasing. Moreover, 
$$
0\leq k^{\frac{2n}{n-2}}|A_k|
\leq \int_{A_k}|u(x)|^{\frac{2n}{n-2}}\,dx
\leq  \int_\Omega |u(x)|^{\frac{2n}{n-2}}\,dx <\infty,
$$
it follows that $|A_k|\to 0$ as $k\to \infty.$ 
Since   $\cc\in L^1(\Omega)$,  the absolute continuity of the Lebesgue integral yields
$
\cc(A_k) \to 0 $ as $  k\to\infty. $
Therefore, 
choosing first   $\varepsilon>0$ sufficiently small and then 
choosing $k_0>1$ sufficiently large, we obtain
\begin{equation}\label{eq-k0}
C\big( \cc(A_k)^{\frac{\epsilon_0}{n-2+\epsilon_0}}  + \varepsilon \big)\leq \frac12, \qquad \forall \ k\geq k_0.
\end{equation}

We now estimate   the boundary  integral $J$ in \eqref{3.5}.  By   the growth condition 
\eqref{1.8},  and using the fact  that $u_k\leq |u|$ on $A_k$,  we  obtain
\begin{equation}\label{eq-psi}
\psi\big(x,u(x)\big) u_k(x)\leq |\psi\big(x,u(x)\big)| u_k(x)
\leq \psi_1(x) |u(x)| + \psi_2(x) |u(x)|^{\beta+1},
\end{equation}
for $\sigma_x$-almost every $ x \in\partial\Omega \cap \overline{A}_k. $ Consequently,
\begin{equation}\label{3.8}
\begin{split}
    J
&\leq \int_{\partial \Omega \cap \overline{A}_k} \psi_1(x) |u(x)|\, d\sigma_x + \int_{\partial \Omega \cap \overline{A}_k} \psi_2(x) |u(x)|^{\beta+1}\, d\sigma_x\\
&=\underbrace{\int_{\partial \Omega \cap \overline{A}_k}  |u(x)|\, d\m_1}_{I_4}  + \underbrace{\int_{\partial \Omega \cap \overline{A}_k}  |u(x)|^{\beta+1}\, d\m_2}_{I_5},
\end{split}
\end{equation}
where  $ d\m_1 = \psi_1(x)\, d\sigma_x $ and $d\m_2= \psi_2(x)\, d\sigma_x$ are positive Radon measures defined on $\partial \Omega.$

For every ball 
 $ B_\rho(x) $  centered at  a point  $ x \in \partial \Omega,$ H\"older's inequality yields   
\begin{align*}
\m_1(B_\rho(x))
&= \int_{B_\rho(x)\cap \partial\Omega} \psi_1(y)\, d\sigma_y
\leq  C(n) \|\psi_1\|_{L^{q_1}(\partial \Omega)} \,\rho^{(n-1)(1-1/q_1)}.
\end{align*}
Since $q_1>n-1,$ by 
Lemma~\ref{trace} applied with $p=r=2$,
 we obtain
\begin{equation}\label{eq-GN}
\| u \|_{L^2(\!\m_1, \partial \Omega \cap \overline{A}_k)}
= \| u \chi_{A_k} \|_{L^2(\!\m_1, \partial \Omega)} \leq C \| u \|_{W^{1,2}(A_k)}^{\tau}
\| u \|_{L^2(A_k)}^{1-\tau},
\end{equation}
where $ C  > 0 $ depends only on $n$ and the measure $\m_1$. Here  
$$
\tau = \frac{n}{2}
-\frac{(n-1)(1-\frac1{q_1})}{2}
=\frac12+\frac{n-1}{2q_1} <1.
$$

Since  $ k > 1 $, we have       
$ |u(x)|>k>1$ for almost every $x\in A_k$  and hence $|u(x)|\leq |u(x)|^2 .$ Consequently 
\begin{equation}\label{eq-I4} 
I_4 \leq \int_{\partial \Omega \cap \overline{A}_k} |u(x)|^2\, d\m_1.
\end{equation}
Applying Young's inequality to \eqref{eq-GN}, and   then  Lemma~\ref{G-N}, we obtain
\begin{equation}\label{eq-u2bound}
\begin{split}
\| u \|&_{L^2(\m_1, \partial \Omega \cap \overline{A}_k)}
\leq  \frac{\varepsilon}{2} \| u \|_{W^{1,2}(A_k)}
+ \frac{C}{\varepsilon^{\tau/(1-\tau)}} \| u \|_{L^2(A_k)} \\
&\leq  \frac{\varepsilon}{2} \|Du\|_{L^2(A_k)} 
+ \left( \frac{\varepsilon}{2}
+ \frac{C}{\varepsilon^{\tau/(1-\tau)}} \right)\|u\|_{L^2(A_k)}\\
&\leq \frac{\varepsilon}{2} \|Du\|_{L^2(A_k)} \\
&+ C\left( \frac{\varepsilon}{2}
+ \frac{C}{\varepsilon^{\tau/(1-\tau)}} \right)\left( \| Du \|_{L^2(A_k)}
+ \llangle u \rrangle_{\sigma,A_k} \right)^t
\llangle u \rrangle_{\sigma,A_k}^{1-t}
\end{split}
\end{equation}
for every $ \varepsilon > 0$ and  $\sigma\in (0,2], $ where
\begin{equation}\label{eq-t}
t = \frac{n(2-\sigma)}{2n - \sigma(n-2)} \in (0,1).
\end{equation}


Applying Young's inequality to the product on the right-hand side of \eqref{eq-u2bound}, we obtain
\begin{equation*}
\begin{split}
\| u \|_{L^2(\!\m_1, \partial \Omega \cap \overline{A}_k)}
&\leq  \varepsilon \|Du\|_{L^2(A_k)} 
 + C(\varepsilon,n,\sigma)
\llangle u \rrangle_{\sigma,A_k}.
\end{split}
\end{equation*}

Therefore, squaring the previous estimate and adjusting the constants if necessary, we obtain
$$
I_4 \leq 
\|u\|_{L^2(\mathfrak{m}_1,
\partial\Omega\cap\overline{A}_k)}^2
\leq
\varepsilon \|Du\|_{L^2(A_k)}^2
+ C_\varepsilon
\llangle u\rrangle_{\sigma,A_k}^2
$$
with a constant $C_\varepsilon>0$ depending only on $\varepsilon, n,$ and $\sigma$.

To  estimate $ I_5 $, we  first estimate the  $\m_2$-measure of a  ball centered at a point  $x\in \partial\Omega$. Indeed
$$
\m_2(B_\rho(x))
= \int_{B_\rho(x) \cap \partial \Omega} \psi_2(y)\, d\sigma_y 
\leq C\|\psi_2\|_{L^{q_2}(\partial \Omega)}\, 
\rho^{(n-1)(1-1/q_2)},
$$
where the constant depends  on  $ n $ and  the Lipschitz character of $\Omega$ but not on $x$ or $\rho$.

Since $ (n-1)(1-1/q_2) > n-2 $,  we can  apply  Lemma~\ref{trace} with $p=2,$ $r=1+\beta$. Moreover, assumption \eqref{1.5}
implies that 
$$
1+\beta < \frac{2(n-1)(1-\frac1{q_2})}{n-2}.
$$
Therefore,
\begin{equation}\label{eq-I5a}
I_5=\| u \|_{L^{1+\beta}(\mathfrak{m}_2, \partial \Omega \cap \overline{A}_k)}^{1+\beta}\leq C \|u\|_{W^{1,2}(A_k)}^{\tau(1+\beta)} \|u\|_{L^2(A_k)}^{(1-\tau)(1+\beta)},
\end{equation}
where  
$$
\tau = \frac{n}{2} - \frac{(n-1)(1-\frac{1}{q_2})}{1+\beta} < 1
$$
and the constant  $C>0 $ depends only  on $n, q_2,$ the measure $\m_2,$ and the Lipschitz character of $\Omega$. 

We now distinguish between  two cases according to the value of $\beta$. Without loss of generality, we assume that $\rho\in(0,1)$. If
$ B_\rho(x)\cap\partial\Omega=\varnothing,$ 
then $  \m_2(B_\rho(x))=0,$ 
and there is nothing to prove. Hence, we may assume that
$ B':=B_\rho(x)\cap\partial\Omega$
is nonempty.

\bigskip

\noindent
\textsc{Case 1:} $ 1+\beta \leq 2 $.

Since $k>1$, we have 
$ |u(x)|^{1+\beta} \leq |u(x)|^2 $ on $ \partial \Omega \cap \overline{A}_k $, and consequently
$$
\int_{\partial \Omega \cap \overline{A}_k}
|u(x)|^{1+\beta} \, d\m_2
\leq
\int_{\partial \Omega \cap \overline{A}_k}
|u(x)|^2 \, d\m_2 .
$$
The right-hand side  can be estimated exactly as in the estimate of $ I_4 $. Hence,  
\begin{equation}\label{Est-I_4b_1}
I_5
\leq
\varepsilon \|Du\|^2_{L^2(A_k)} 
+ C_\varepsilon\,
\llangle u \rrangle_{\sigma,A_k}^{\,2},
\end{equation}
for every $\varepsilon>0$,  every $\sigma\in(0,2]$ and constant $C_\varepsilon>0$ depending only on $\varepsilon, n,$ and~$\sigma.$

\bigskip
\noindent
\textsc{Case 2:} $ 1+\beta > 2 $.

In this case, we proceed by rewriting the right-hand side of \eqref{eq-I5a}. To this end, choose a real number $\theta$ satisfying
$$
\max\left\{
\frac{1}{2},
\frac{1}{\tau(1+\beta)}
\right\}<\theta < \frac{1}{\tau(1+\beta) -\beta +1}.
$$
Such a choice  is  possible since $ 0< \tau(1+\beta) -\beta +1<2,$ 
which implies
$$
\tau(1+\beta) - \frac{1}{\theta} > 0,
\qquad
(1-\tau)(1+\beta) - \frac{2\theta-1}{\theta} > 0 .
$$

Hence, we decompose the right-hand side of \eqref{eq-I5a} as follows
$$
I_5\leq 
C\|u\|_{W^{1,2}(A_k)}^{\frac{1}{\theta}}
\,\|u\|_{L^2(A_k)}^{\frac{2\theta-1}{\theta}} 
\,\|u\|_{W^{1,2}(A_k)}^{\tau(1+\beta)-\frac{1}{\theta}}
\,\|u\|_{L^2(A_k)}^{(1-\tau)(1+\beta)-\frac{2\theta-1}{\theta}} .
$$

Since  $u\in W^{1,2}(\Omega),$  it holds that 
$$
\|u\|_{L^2(A_k)} \leq \|u\|_{W^{1,2}(A_k)} \leq \|u\|_{W^{1,2}(\Omega)}<\infty.
$$
Therefore
$$
\|u\|_{W^{1,2}(A_k)}^{\tau(1+\beta)-\frac{1}{\theta}}
\|u\|_{L^2(A_k)}^{(1-\tau)(1+\beta)-\frac{2\theta-1}{\theta}}
\leq
\|u\|_{W^{1,2}(\Omega)}^{\beta-1} =:N(u)
$$
where $N(u)$  depends only on $u$ and is independent of  $k$.

Therefore,
$$
I_5 \leq C\,N(u)
\|u\|_{W^{1,2}(A_k)}^{\frac1\theta}
\|u\|_{L^2(A_k)}^{\frac{2\theta-1}{\theta}}.
$$
Since $2\theta>1$,
 Young's inequality yields
\begin{align*}
I_5
&\leq \frac{\varepsilon}{2}
\|u\|_{W^{1,2}(A_k)}^2 
+ C(\varepsilon,N(u)) \|u\|_{L^2(A_k)}^2\\
&\leq \frac{\varepsilon}{2}\|Du\|^2_{L^2(A_k)}+
\left(  \frac{\varepsilon}{2} +C(\varepsilon,N(u)) \right)\|u\|^2_{L^2(A_k)}.
\end{align*}

Interpolating the $L^2$-norm of $u$ by means of  Lemma~\ref{G-N} with $p=r=2$, followed by  Young's inequality,  we obtain
\begin{align*}
\|u\|_{L^2(A_k)}^2
&\leq
C \Big( \|Du\|_{L^2(A_k)} +
\llangle u\rrangle_{\sigma,A_k}
\Big)^{2t} \llangle u\rrangle_{\sigma,A_k}^{2(1-t)}\\
&\leq \delta \Big(
\|Du\|_{L^2(A_k)} + \llangle u\rrangle_{\sigma,A_k}
\Big)^2 + C(\delta)
\llangle u\rrangle_{\sigma,A_k}^{2}\\
&\leq  2\delta\,\|Du\|_{L^2(A_k)}^2 + C(\delta)\,
\llangle u\rrangle_{\sigma,A_k}^{2}.
\end{align*}

Substituting the above estimate for
$ \|u\|_{L^2(A_k)}^2$ 
into the previous inequality, and choosing  $\delta >0 $ sufficiently small,   we infer
 \begin{align*}
I_5\leq \varepsilon \|Du\|^2_{L^2(A_k)}
+C(\varepsilon, N(u)) \llangle u \rrangle^2_{\sigma,A_k}.
 \end{align*}

Combining the estimates for $I_4$ and $I_5$,  recalling \eqref{3.8}, and adjusting the constants if necessary,  we conclude that
$$
J\leq \varepsilon  \|Du\|^2_{L^2(A_k)}  + C(\varepsilon, N(u))
\llangle  u\rrangle_{\sigma,A_k}^2.
$$

Since $u=u_k+k$ on $A_k$ and   $ Du(x) = Du_k(x) $ for almost every  $ x \in A_k $, we have
\begin{equation}\label{eq-uk}
\begin{split}
\llangle u\rrangle_{\sigma,A_k}^2 
&= \llangle u_k+k\rrangle_{\sigma,A_k}^2 
\leq C_\sigma^2 \big(\llangle u_k\rrangle_{\sigma,A_k}^2 
+  k^2 |A_k|^{2/\sigma}\big)\\
&\leq C\big( \llangle u_k\rrangle_{\sigma,A_k}^2 
+ k^2\m(A_k)\big),
\end{split}
\end{equation}
where we have used the fact that $\sigma\in(0,2]$, so that
$2/\sigma\ge1$, and therefore
$$
|A_k|^{2/\sigma}=
|A_k|\,|A_k|^{2/\sigma-1} \leq
|\Omega|^{2/\sigma-1}|A_k| \leq
C\,\mathfrak m(A_k),
$$
since $\Omega$ is bounded. The constant $C_\sigma$ is given  in Remark~\ref{rem1}.

Summarizing the above estimates, after adjusting the constants if  necessary,  and recalling that
$Du=Du_k$ almost everywhere in $A_k$, we obtain
\begin{equation}\label{Est-I_4}
J\leq \varepsilon  \|Du_k\|^2_{L^2(A_k)} 
+ C(\varepsilon,N(u))
\left(\llangle u_k\rrangle_{\sigma,A_k}^2 
+  k^2 \m(A_k)\right),
\end{equation}
for every $k\ge1$, every $\varepsilon>0$, and every $\sigma\in(0,2]$.

The previous estimates allow us to control the energy of the weak solution on the upper level set $A_k$. 
Substituting the estimates \eqref{3.13} and \eqref{Est-I_4} into  \eqref{3.5} 
and choosing $k\geq k_0$ and $\varepsilon$, so that \eqref{eq-k0} holds, we obtain
\begin{equation}\label{3.13'}
 \|Du_k\|_{L^2(A_k)}^2 
\leq C\left(
k^2 \m(A_k)
+ \llangle u_k\rrangle_{\sigma,A_k}^2\right), \qquad \forall \, k\geq k_0.
\end{equation}
Moreover,  inequality \eqref{3.4} becomes
\begin{equation}\label{3.14}
    \int_{A_k} u_k(x)\, d\m
\leq  C \m(A_k)^{1-\frac{n-2}{2(n-2+\varepsilon_0)}}  \left(k \m(A_k)^{1/2}
+ \llangle u_k\rrangle_{\sigma,A_k} 
\right),
\end{equation}
for every $ k\geq k_0 $ and every $ \sigma\in(0,2] $.

Our next  goal  is  to show that, for  a suitable choice of $ \sigma\in(0,2] $, the estimate 
\begin{equation}\label{3.15}
\llangle u_k\rrangle_{\sigma,A_k}
\leq C\, k \m(A_k)^{1/2},
\end{equation}
holds, where the constant $C$ depends only on the prescribed data.

We 
choose $\sigma$ in the following way 
$$
\begin{cases}
    \sigma\in \left(0, \frac{4}{n-2}\right], & \text{ if } n\geq 4\\[4pt]
    \sigma \in (0,2],  & \text{ if } 2<n<4.
\end{cases}
$$
Then $\sigma\leq \frac{4}{n-2}$, and  by  H\"older's  inequality and the definition \eqref{eq-measure} of the measure $\m,$ we obtain
\begin{align*}
\int_{A_k} &|u_k(x)|^\sigma \, dx 
\leq C\left(\int_{A_k} |u(x)|^\sigma\, dx
+k^\sigma \m(A_k)\right)\\
&\leq C\left[\left(\int_{A_k} |u(x)|^{\frac{4}{n-2}} \,dx
\right)^{\frac{(n-2)\sigma}{4}} \m(A_k)^{1- \frac{(n-2)\sigma}{4}}+ k^\sigma \m(A_k)\right]\\
&\leq  C[\m(A_k) +k^\sigma \m(A_k)]    \leq C k^\sigma \m(A_k),
\end{align*}
since $k>1$.

Since $\mathfrak m$ is absolutely continuous with respect to the Lebesgue measure and $|A_k|\to0$ as $k\to\infty$, we have $\mathfrak m(A_k)\to0$. Hence, choosing $k_0$ sufficiently large so that $\m(A_k)<1$ for every $k\geq  k_0$, and recalling that $1/\sigma\ge1/2$, we obtain
\begin{equation}\label{eq-mAk1}
\llangle u_k\rrangle_{\sigma,A_k} \leq
Ck\,\m(A_k)^{1/\sigma}
\leq Ck\,\m(A_k)^{1/2},
\end{equation}
which proves \eqref{3.15}.

Substituting  \eqref{eq-mAk1} into  \eqref{3.14}, we obtain
\begin{equation}\label{3.16}
\int_{A_k} u_k(x)\, d\m
\leq C k \m(A_k)^{1+\frac{\varepsilon_0}{2(n-2+\varepsilon_0)}}
\qquad \forall\  k\geq k_0.
\end{equation}

By Cavalieri's principle,
$$
\int_{A_k} u_k(x)\, d\m
= \int_k^\infty \m(A_t)\, dt.
$$
Setting $ \tau(t):=\m(A_t) $,  inequality \eqref{3.16} can be rewritten  as
$$
\int_k^\infty \tau(t)\, dt
\leq C k \tau(k)^{1+\delta},
\qquad
\delta=\frac{\varepsilon_0}{2(n-2+\varepsilon_0)}>0.
$$

Since $\tau$ is non-increasing, all the assumptions of the 
 Hartman--Stampacchia lemma
(Lemma~\ref{lem2.5}) are satisfied. Hence, there exists a constant  $k_{\max}>0,$ depending  only on the prescribed data and on 
$ \|u\|_{L^2(\Omega)} $ and $ \|Du\|_{L^2(\Omega)}$, through the quantity  $N(u)$, such that
$$
u(x)\leq k_{\max}
\qquad \text{ for almost every  } \  x\in \Omega.
$$

Applying the same arguments to $ -u $, we obtain a corresponding lower bound. 
Hence $u\in L^\infty(\Omega),$
which proves \eqref{eq-max} in the case  $ n>2 $.

The proof in the limiting  case $n=2$ is considerably simpler, since the Sobolev embedding $  W^{1,2}(\Omega)\hookrightarrow L^q(\Omega) $ 
holds for every finite exponent $q>2$. 
Thus,  every occurrence of the critical Sobolev exponent $2^*$ 
 is replaced by an arbitrary exponent $q>2.$ 

Because of Lemma \ref{lem2.4}, we have $Du\in L^{m_0}(\Omega)$ for some
$m_0>2$. Hence,
$  u\in W^{1,m_0}(\Omega).$ 
By the Sobolev embedding theorem,
$$
W^{1,m_0}(\Omega)\hookrightarrow C^{0,\alpha}(\overline{\Omega}),
\qquad
\alpha=1-\frac{2}{m_0}>0.
$$
Therefore, $u$ is Hölder continuous on $\overline\Omega$, and in particular,
$u\in L^\infty(\Omega)$. This   completes the proof.

\end{proof}

  \section*{Acknowledgements}

\textit{R. Rescigno}, \textit{L. Softova}, and \textit{S. Tramontano} are members of INdAM-GNAMPA.

The authors are grateful to the referee for the careful reading of the manuscript and for the valuable comments and suggestions, which led to a significant improvement of the paper.

The authors declare that they have no conflicts of interest.

\end{document}